\documentclass{amsart}
\usepackage{amsfonts}
\usepackage{amscd}
\usepackage{amssymb}

\numberwithin{equation}{section}

\theoremstyle{plain}

\newtheorem{theorem}{Theorem}[section]
\newtheorem{proposition}[theorem]{Proposition}
\newtheorem{lemma}[theorem]{Lemma}
\newtheorem{corollary}[theorem]{Corollary}
\newtheorem{conjecture}[theorem]{Conjecture}
\newtheorem{definition}[theorem]{Definition}

\theoremstyle{definition}

\newtheorem{remark}[theorem]{Remark}

\newcommand{\R}{{\mathbb R}}

\newcommand{\C}{{\mathbb C}}
\newcommand{\N}{{\mathbb N}}
\newcommand{\T}{{\mathbb T}}
\newcommand{\F}{{\mathcal F}}
\newcommand{\supp}{{\operatorname {supp}\,}}

\renewcommand{\H}{{\mathcal H}}
\renewcommand{\S}{{\mathcal S}}

\renewcommand{\Im}{{\mathop{\mathrm{Im}}\,}}

\begin{document}
\title[Real Paley--Wiener theorems and local spectral radius formulas]
{{Real Paley--Wiener theorems and local spectral radius formulas}}
\author{Nils~Byrial~Andersen}
\address{Nils~Byrial~Andersen,
Alssundgymnasiet S\o nderborg,
        Grundtvigs All\'e 86,
        6400 S\o nderborg,
        Denmark}
\email{nb@ags.dk}
\thanks{The first author was supported by a research
grant from the {\it{European Commission}}
IHP Network: 2002--2006
{\it{Harmonic Analysis and Related Problems}}
(Contract Number: HPRN-CT-2001-00273 - HARP)}
\author{Marcel de Jeu}
\address{Marcel~de~Jeu, Mathematical Institute,
Leiden University,
P.O. Box 9512,
2300 RA Leiden,
The Netherlands}
\email{mdejeu@math.leidenuniv.nl}
\date{18 April, 2008}

\begin{abstract}
We systematically develop real Paley--Wiener theory for the Fourier transform on $\R^d$ for Schwartz functions,
$L^p$-functions and distributions, in an elementary treatment based on the inversion theorem. As an application, we show how versions of classical Paley--Wiener theorems can be derived from the real ones via an approach which does not involve domain shifting and which may be put to good use for other transforms of Fourier type as well. An explanation is also given why the easily applied classical Paley--Wiener theorems are unlikely to be able to yield information about the support of a function or distribution which is more precise than giving its convex hull, whereas real Paley--Wiener theorems can be used to reconstruct the
support precisely, albeit at the cost of combinatorial complexity. We indicate a possible application of real Paley--Wiener theory to partial differential equations in this vein and furthermore we give evidence that a number of real Paley--Wiener results can be expected to have an interpretation as local spectral radius formulas. A comprehensive overview of the literature on real Paley--Wiener theory is included.
\end{abstract}
\maketitle

\section{Introduction and overview}\label{sec:intro}

A Paley--Wiener theorem is a characterization, by relating support to growth, of the image of a space of functions
or distributions under a transform of Fourier type.
Starting with the original Paley--Wiener theorem, \cite[Theorem X]{PW}, which describes the
Fourier transform of $L^2$-functions on the real line with support in a symmetric interval
as entire functions of exponential type whose restriction to the real line are $L^2$-functions, such
results have proven to be a basic tool for transform in various set-ups. As a familiar example, if $f$ is a smooth
compactly supported function on $\R^d$ with Fourier transform $\F f:\R^d\to\C$, then it is easily seen that $\F f$ extends
to an entire function on $\C^d$ with the property that, for each $n=0,1,2,\ldots$, there exists a constant $C_n$, such that
\begin{equation}\label{eq:PWestintro}
|\F f(z)|\leq C_n (1+|z|)^{-n} e^{H_A(\Im z)}\qquad(z\in\C^d),
\end{equation}
where $A$ is the convex hull of the support of $f$ and $H_A:\R^d\to\R$ is its supporting function, defined in terms of the standard inner
product as $H_A(x)=\max_{a\in A}a\cdot x$, for $x\in\R^d$. The non-trivial part of the Paley--Wiener theorem for smooth functions is the
converse: if $A$ is a compact and convex subset of $\R^d$, and if an entire function $g:\C^d\to\C$ satisfies estimates as
in \eqref{eq:PWestintro}, then $g$ is the Fourier transform of a smooth function supported in $A$. Whereas this is all common
knowledge, it seems to be less widely known that $H_{\textup{co}(\supp f)}$ -- where $\textup{co}(\supp f)$ is the convex hull
of the support of $f$ -- does not just give a theoretical upper bound for the dominant part of the growth of $\F f$ on $\C^d$
as in \eqref{eq:PWestintro}, but actually prescribes it: the exponential function in \eqref{eq:PWestintro} is really needed,
and precisely in this form. This observation goes back to Plancherel and P{\'o}lya \cite{PlPo}; there is also a detailed account
available as \cite[Theorem~3.4.2]{RoEntire}. To be more concrete, their result, when combined with \cite[Theorem~3.4.3]{RoEntire},
implies that, if $y\in\R^d$ is fixed,
\begin{equation}\label{eq:prescribed}
\limsup_{t\to\infty} \frac{1}{t} \log|\F f(x+ity)|=H_{\textup{co}(\supp f)}(y),
\end{equation}
for almost all $x$ in $\R^d$. Consequently, if $f_1$ and $f_2$ are smooth compactly supported functions such that the convex hulls of their
supports are equal, and if $y\in\R^d$ is fixed,
then $\limsup_{t\to\infty} \frac{1}{t} \log|\F f_1(x+ity)|=\limsup_{t\to\infty} \frac{1}{t} \log|\F f_2(x+ity)|$ for
almost all $x$ in $\R^d$. Informally speaking, this means that the growth of $\F f$ on $\C^d$ enables one to retrieve
the convex hull of the support of $f$, but that more precise information can not readily be obtained from it. Hence the Paley--Wiener
theorem is optimal in this respect. This phenomenon is not limited to smooth functions: \cite{PlPo} was already concerned with the more
general $L^2$-case and it is known \cite[Theorem~5.1.2]{RoRegular}, \cite{Napalkov}, that the straightforward analogue of \eqref{eq:prescribed}
is true for a general compactly supported distribution $T$. Thus the dominant part of the growth of $\F T$ on $\C^d$ is prescribed by the convex
hull of the support of $T$, irrespective of the precise support or the degree of regularity of $T$, and the Paley--Wiener theorem for
distributions with compact support is again optimal in this respect.

Whereas in view of the above there seems to exist a fundamental theoretical obstruction for the classical Paley--Wiener theorems to ``look inside'' the
convex hull of the support, this is not the case for a new type of theorems which have enjoyed increasing interest in recent years.
They have become known as ``real Paley--Wiener theorems'', in which the adjective ``real'' expresses that
information about the support of $f$ comes from growth rates associated to the function $\F f$ on $\R^d$, rather than on $\C^d$ as in the
classical ``complex Paley--Wiener theorems''. A typical example is the following: If $P$ is a polynomial with corresponding constant
coefficient differential operator $P(\partial)$, $f$ is a Schwartz function on $\R^d$, and $1\leq p\leq\infty$, then in the extended
positive real numbers
\begin{equation}\label{eq:limSchwartzintro}
\lim _{n\to \infty}\| P(\partial )^n f\| _p ^{1/n} =\sup _{\lambda \in \supp \F f} |P(i\lambda)|.
\end{equation}

Thus, as compared to the complex theorems, the real theorems involve a growth rate as $n\to\infty$ rather
than $|z|\to\infty$.\footnote{The roles of $f$ and $\F f$ in the literature on real Paley--Wiener theorems $\R^d$
are reversed as compared to the complex theorems, so that some mental gymnastics is inevitable, but this is not a
consequence of unfortunate choices. Both roles have their place in general Lie theory, and it is only
because $\R^d$ happens to be self-dual that the results for the group and the unitary dual intermingle at equal
footing. See also \cite{AdJ2}.} What is more important from a theoretical point of view: if an upper bound $M$ for the
left hand side in \eqref{eq:limSchwartzintro} is
known, then one concludes that $\supp \F f$ is contained in $\{\lambda\in\R^d : |P(i\lambda)|\leq M\}$, which need be neither compact
nor convex. In fact, the complex theorem will not even apply if $\F f$ does not have compact support, although the real theorem still does.
It is in this fashion that more information about the support can be extracted compared to using the complex theorems, albeit at the cost of
combinatorial complexity. One can even reconstruct the support itself in a number of cases. As an example, it is one of our results
that, if $1\leq p\leq\infty$, $f$ is an $L^p$-function and $K$ is compact, then the support of the distribution $\F f$ is contained in $K$
if, and only if,
\begin{equation}\label{eq:reconstruct}
\lim_{n\to \infty}\| P(\partial )^n  f\| _p ^{1/n}  \le \sup _{k\in K} |P (ik)|,
\end{equation}
for all polynomials $P$.\footnote{It follows from the hypotheses that $P(\partial )^n  f$ is in $L^p$ for all $n$.} In Section~\ref{sec:realPW}, it is also explained how it sometimes is
sufficient for this reconstruction to control $ \| P(\partial )^n  f\| _p ^{1/n}$ for
sufficiently many polynomials, rather than for all of them. We have similar results for general tempered distributions.

Our results in real Paley--Wiener theory are contained in Section~\ref{sec:realPW} below. We
consider pointwise estimates as well as $L^p$-norms; Schwartz functions as well as $L^p$-functions
and distributions; and single polynomial results such as \eqref{eq:limSchwartzintro}
as well as multiple polynomial results such as \eqref{eq:reconstruct}. The reader who consults
Bang's work, \cite{Ba1,Ba3}, and Tuan's work, \cite{Tu1,Tu2}, will see that our work is related to these
papers, but that, where there is overlap, our results are more general and that the proofs are
quite different. As an example, \eqref{eq:limSchwartzintro} for Schwartz functions was first
established by Bang \cite{Ba1}, in one dimension and with $P(x)=x$, using the complex Paley--Wiener
theorem, and later by Tuan, in arbitrary dimension for real polynomials, using the Plancherel
theorem \cite{Tu2}. In our paper we systematically base our proofs on the inversion theorem,
and this yields \eqref{eq:limSchwartzintro} for arbitrary polynomials, as well as its analogue
for $L^p$-functions which was hitherto not within reach. Our proofs are also considerably more elementary,
especially compared to the proofs using Sobolev theory for elliptic equations in \cite{Ba3}.
By giving a systematic broad development of the theory, based on simple proofs and extending as
well as generalizing previous insights, we hope to provide a basis for future developments -- also for other transforms -- of this theory which was initiated by Bang and Tuan.

It is noteworthy that real Paley--Wiener theorems can provide a new and alternative method of proof
for complex Paley--Wiener theorems. This is illustrated in Section~\ref{sec:complexPW} below, where
we derive complex Paley--Wiener theorems for the Fourier transform from the real ones in Section~\ref{sec:realPW}.
One does not obtain the strongest possible forms of the complex theorems for the Fourier transform in this way, but
the important feature of this approach is that it does not involve shifting of the domain, as in the usual proofs of
the complex results. This may work for other integral transforms as well and, in fact, in \cite{AdJ1} it was
demonstrated how this idea can be employed fruitfully to prove a complex Paley--Wiener theorem for the Dunkl transform
in one dimension, where the idea of contour shifting does simply not apply since the integrand is not entire.

We have included a comprehensive overview of the literature on real Paley--Wiener theory in Section~\ref{sec:comparison}, in which we discuss the most important contributions by Bang, Tuan and others, and compare them to our results. This section contains a fair number of references and, although we do not claim completeness, we expect that the material in this section will be useful to anyone who wants to gain an overview of the field at this time.

The final Section~\ref{sec:perspectives} contains perspectives for future developments.
One of these, and an intriguing one in our opinion, is a possible interpretation of equations such as \eqref{eq:limSchwartzintro} in local spectral theory.
Namely, it may well be that the right hand side of \eqref{eq:limSchwartzintro} is the radius of a local spectrum, so that \eqref{eq:limSchwartzintro} is a manifestation of a local spectral radius formula for an unbounded operator. For $p=1$ we are, in fact, able to prove this and we have established similar results for compact connected Lie groups \cite{AdJ2}. Since we are not aware of general results on the a priori validity of a local spectral radius formula for unbounded operators (quite contrary to the bounded case), and since this also gives a new angle on real Paley--Wiener theorems, we discuss this subject in some detail. We are indebted to Jan van Neerven for suggesting the possibility of a connection with local spectral theory.

\bigskip

\textit{Notations and preliminaries.} Our notations are the usual ones, as in \cite{Hoer} or \cite[Chapter 7]{Ru},
where also the proofs of the properties mentioned below can be found. We let $\N=\{1,2,3,\ldots\}$ and $\N_0=\{0,1,2,\ldots\}$.
Furthermore, $\S(\R^d)$ stands for the Schwartz space on $\R^d$ with dual $\S '(\R^d)$, the space of tempered distributions on $\R^d$.
Recall that $L^p(\R^d) \subset \S '(\R^d)$, for all $1\le p\le \infty$.
If $P$ is a polynomial on $\R^d$, we let $P(\partial)$ denote the corresponding differential operator with constant coefficients,
so that $P(\partial)e^{i \lambda\cdot x} = P(i\lambda) ^{i \lambda\cdot x}$,
for all $\lambda,x$. For $T\in \S'(\R^d)$, the distribution $P(\partial) T\in  \S'(\R^d)$ is defined by
\begin{equation*}
\langle  P(\partial)f,  \phi \rangle =\langle  f,  P(-\partial) \phi \rangle
\qquad (\phi \in \S(\R^d)),
\end{equation*}
which is compatible with the action on smooth functions.
The convolution $f*g$ of two functions $f,g\in L^1(\R^d)$ is defined as
\begin{equation}\label{conv}
f*g\,(x)=\int _{\R^d}  f (x-y) g(y) dy \qquad (x\in \R^d).
\end{equation}
If $f\in L^p(\R^d),\,1\le p\le \infty$, and $g\in  L^1(\R^d)$, then the integral (\ref{conv})
converges for almost all $x\in \R^d$, and we have the following inequality
\begin{equation*}
\|f*g\|_p\le \|f\|_p \|g\|_1  .
\end{equation*}
We also notice that
\begin{equation*}
P(\partial ) (f*\phi ) =f* P(\partial ) \phi ,
\end{equation*}
for all $f\in L^p(\R^d),\,1\le p\le \infty$, and $\phi\in  \S(\R^d)$.

Our normalization of the Fourier transform $\F f$ of a function $f\in L^1(\R^d)$ is as
\begin{equation*}
\F f (\lambda) =\frac{1}{(2\pi)^{d/2}} \int _{\R^d} f(x) e^{-i \lambda\cdot x}\, dx \qquad (\lambda \in \R^d),
\end{equation*}
where $\lambda \cdot x  = \sum _{i=1}^d \lambda _i x_i$. The inversion formula then holds with the same constant $(2\pi)^{-d/2}$
and the Plancherel theorem is valid with Lebesgue measure on both sides.
The Fourier transform $\F T$ of a tempered distribution $T$ is defined
\begin{equation}\label{distfouriertransform}
\langle \F T,  \phi \rangle =\langle  T,  \F \phi \rangle \qquad (\phi \in \S(\R^d)),
\end{equation}
which is compatible with its definition on $L^1(\R^d)$. Then, for all $T\in\S' (\R^d)$,
\begin{equation}\label{intertwining}
\F (P(\partial ) T) (\lambda) =  P(i\lambda) \F T (\lambda)\qquad (\lambda \in \R).
\end{equation}
Also,
\begin{equation*}
\F (f*\phi) = \F f \cdot \F \phi \qquad (f\in L^p(\R^d),\,\phi\in  \S (\R^d)).
\end{equation*}

It is with some emphasis (see the discussion in Section~\ref{sec:comparison} on the preference for proofs based on the
inversion theorem over those based on the Plancherel theorem), that we mention that all polynomials are complex valued unless otherwise stated.

\section{Real Paley--Wiener theorems}\label{sec:realPW}

Let $f$ be a measurable function on $\R^d$ representing a tempered distribution, and let $P$ be a polynomial.
For $f$ in various classes, this section is concerned with the relation between the growth behavior of the sequence
$\{P(\partial)^n f\}_{n=0}^\infty$ on $\R^d$ and the supremum of $|P(i\lambda)|$ on the support of $\F f$.
The notation for quantities such as the latter is introduced in the following definition.

\begin{definition} Let $T$ be a distribution on $\R^d$ and $P$ a polynomial. Then we let
\[
R(P,T)=\sup\{|P(i\lambda)| : \lambda\in\supp T\}\in\R_{\geq 0}\cup\{\infty\},
\]
where by convention $R(P,T)=0$ if $T=0$.
\end{definition}

A typical statement in the sequel will be that $R(P,\F f)\leq R$, for some $R\geq 0$, or, equivalently,
that $\supp \F f\subset\{\lambda\in\R^d : |P(i\lambda)|\leq R\}$. We will also be concerned with reconstructing $\supp \F f$,
given the knowledge that it is in sufficiently many of such polynomially defined sets.

Our first result in this vein is the following real Paley--Wiener theorem for Schwartz functions. The corresponding
results for tempered distributions in general are included as Propositions~\ref{prop:realPWtempdist1}
and~\ref{prop:realPWtempdist2}, and Theorem~\ref{thm:realPWcompdist}.

\begin{theorem}\label{thm:realPWSchwartz}
Let $P$ be a polynomial, $f\in \S(\R^d)$, and $0< R< \infty$.
Then the following are equivalent:
\begin{enumerate}
\item[(a)] There exists a function $\phi : \N \to \R_+$ such that
\begin{equation*}
\liminf _{n\to \infty}\phi (n)^{1/n} \ge 1/R
\end{equation*}
and
\begin{equation*}
\sup _{n\in \N} \phi (n) \| (1+ |x|)^{d+1} P(\partial)^n f\| _\infty < \infty;
\end{equation*}
\item[(b)] For each $N\in \N_0$, there exists a constant $C$, such that, for all $n\in\N$ and $x\in\R^d$,
\begin{equation*}
|P(\partial)^n f(x)|\leq C n^N R^n(1+|x|)^{-N}
\end{equation*}
\item[(c)] $R(P,\F f) \le R.$
\end{enumerate}
\end{theorem}
\begin{proof}
First we prove that (a) implies (c). Assume that $|P(i\lambda_0)| \ge R+ \varepsilon$ for some $\lambda _0\in \R^d$ and $\varepsilon >0$.
The implication will follow once we show that $\F f(\lambda _0)=0$.
For $n\in \N$, one has
\begin{equation*}
\F f  (\lambda _0) =\frac{1} {(2\pi)^{d/2}P(i\lambda_0)^n} \int _{\R^d} (P(\partial)^n f)(x) e^{-i\lambda_0\cdot x}\, dx,
\end{equation*}
hence, for all $n\in \N$ large enough (so that $\phi(n)\neq 0$),
\begin{align*}
|\F f  (\lambda _0) |&\le\frac {1}{(2\pi)^{d/2}|P(i\lambda_0)|^n}
\int _{\R^d}(1+|x|)^{-d-1} |(1+|x|)^{d+1} (P(\partial)^n f)(x) |\, dx\\
&\le C \phi (n) ^{-1}
|P(i\lambda_0)|^{-n},
\end{align*}
for some positive constant $C$. So
\begin{equation*}
|\F f  (\lambda _0) |\le C\limsup_{n\to \infty}(  \phi (n) ^{-1} |P(i\lambda_0)|^{-n}) \le C \limsup_{n\to \infty}
\frac {(R+\varepsilon/2) ^n}{(R+\varepsilon)^n}  =  0.
\end{equation*}

Now we prove that (c) implies that
\begin{equation}\label{equivalentestimates}
\sup_{n\in\N} R^{-n}n^{-N} \Vert |x|^N P(\partial) f \Vert_\infty <\infty
\end{equation}
for each $N\in\N_0$, which is easily seen to be equivalent to (b). In order to establish \eqref{equivalentestimates}, we will first show that
\begin{equation}\label{preparatoryestimates}
\sup_{n\in\N}R^{-n}n^{-N}\Vert (y\cdot x)^N P(\partial)^n f\Vert_\infty <\infty
\end{equation}
for each $N\in\N_0$ and $y\in\R^d$ (where the supremum norm is taken of the function of $x$).
Fix such $N$ and $y$. Then, for all $n\in\N$ and all $x\in\R^d$, we have
\begin{align}\label{rewrittenasintegral}
(y\cdot x)^N P(\partial)^n f(x)&=\frac{i^N}{(2\pi)^{d/2}}\int_{\R^d}\partial_y^N\left[P(i\lambda)^n\F f(\lambda)\right]e^{i\lambda\cdot x}\,d\lambda\\
&=\frac{i^N}{(2\pi)^{d/2}}\int_{\R^d}\left ( \sum_{k=0}^N \binom{N}{k}\partial_y^k P(i\lambda)^n\,\partial_y^{N-k} \F f(\lambda)\right )\,e^{i\lambda\cdot x}\,d\lambda.\notag
\end{align}
An induction with respect to $k$ shows that
\begin{align*}
\partial_y^k P(i\lambda)^n&=\sum_{l=0}^k n(n-1)\cdots(n-l+1)P_{l,k}(\lambda)P(i\lambda)^{n-l}\\
&=P(i\lambda)^{n-k}\sum_{l=0}^k n(n-1)\cdots(n-l+1)P_{l,k}(\lambda)P(i\lambda)^{k-l},
\end{align*}
for some polynomials $P_{l,k}$ independent of $n$. Here we have assumed that $n>N$,
so that $n-k>0$ for all $k$ that occur in the summation, hence $P(i\lambda)^{n-k}$ is
defined. This implies that, for $k$ occurring in the summation, $n>N$, and $\lambda\in\sup \F f$,
\begin{align*}
|\partial_y^k P(i\lambda)^n|&\leq R(P,\F f)^{n-k}n^N\sum_{l=0}^k |P_{l,k}(\lambda)P(i\lambda)^{k-l}|\\
&\leq R^{n}n^N\sum_{l=0}^k R^{-k} |P_{l,k}(\lambda)P(i\lambda)^{k-l}|
\end{align*}
and using this in \eqref{rewrittenasintegral} shows that \eqref{preparatoryestimates}
holds with the supremum taken over $n>N$, which implies \eqref{preparatoryestimates} itself.

To conclude the proof that (c) implies (b), we choose a basis $y_1,\ldots,y_d$ of $\R^d$.
Then there exists $C>0$, such that $|x|\leq C\max_j|y_j\cdot x|$, for all $x\in\R^d$,
hence $|x|^N\leq C^N\sum_{j=1}^d |y_j\cdot x|^N$. Using this and applying \eqref{preparatoryestimates}
to the $y_j$ shows that \eqref{equivalentestimates} holds.

Finally, to prove that (b) implies (a), we simply take $\phi (n) = R^{-n}n^{-d-1}$.
\end{proof}

\begin{remark}
Controlling $P(\partial)^n f$ as in (a) or (b) implies that
$\supp \F f\subseteq\{\lambda\in\R^d : |P(i\lambda)|\leq R\}$. As the example $P(x_1,x_2) = x_1x_2$ on $\R^2$ shows,
this latter set can be non-compact and non-convex. This illustrates the point, mentioned in the introduction, that
the theoretical limitations to compact and convex sets which seems to exist for complex Paley--Wiener theorems, are
not present in the real case.
\end{remark}

We will now study the $L^p$-case, for which we need the following key proposition.

\begin{proposition}\label{prop:Lpkeyprop}
Let $P$ be a polynomial and $1\le p\le \infty$. Suppose $P(\partial)^n f\in L^p(\R^d)$, for all $n\in \N_0$.
Then in the extended positive real numbers
\begin{equation} \label{liminflimit}
\liminf _{n\to \infty}\| P(\partial )^n f\| _p ^{1/n} \ge R(P,\F f).
\end{equation}
\end{proposition}
\begin{proof}
Fix $\lambda _0\in \supp \F f$. We can assume that $P(i\lambda _0) \ne 0$.
We will show that
\begin{equation}\label{mustshow2}
\liminf _{n\to \infty}\| P(\partial )^n f\| _p ^{1/n} \ge |P(i\lambda _0)| - \varepsilon,
\end{equation}
for any (fixed) $\varepsilon >0$ such that
$0 < 2\varepsilon < |P(i\lambda _0)|$.

To this end, choose and fix $\psi \in C^\infty _c (\R^d)$ such that
$\langle \F f, \psi \rangle\ne 0$, and
\begin{equation}\label{compact}
\supp \psi \subset \{\lambda \in \R^d : |P(i\lambda _0)| - \varepsilon < |P(i\lambda )|  < |P(i\lambda _0)| + \varepsilon\}.
\end{equation}
This is possible since $\lambda _0$ belongs to the set in (\ref{compact}), and $\lambda _0 \in \supp \F f$.
For $n\in \N_0$, let $\psi _n (\lambda) = P(i\lambda)^{-n} \psi (\lambda)$.
Let $\frac 1p +\frac 1q=1$, and fix $M\in \N$ such that $(1+|x|^2)^{-M}\in L^q (\R^d)$.
We note that
\begin{equation}\label{intexpression}
(1+|x|^2)^{M} \F \psi _n (x) = \frac{1}{(2\pi)^{d/2}} \int _{\R^d} (1- \Delta)^M \{ P(i\lambda)^{-n} \psi (\lambda)\}e^{-i\lambda\cdot x} \,d\lambda,
\end{equation}
where $\Delta$ is the Laplacian.
When expanding $(1- \Delta)^M \{ P(i\lambda)^{-n} \psi (\lambda)\}$, it can be written as a finite sum
\begin{equation*}
\sum_{\alpha,\beta} c_{\alpha,\beta} (D^\alpha (P(i\lambda)^{-n})) (D^\beta \psi),
\end{equation*}
for coefficients $c_{\alpha,\beta}$, and
composites $D^\alpha ,D^\beta$ of partial derivatives of order $|\alpha|,|\beta|\le2 M$.
Neither coefficients nor the differential operators depend on $n$.

Let $n>2M$. Using induction with respect to $|\alpha|$, it is seen that, if $P(i\lambda)\neq 0$,
\begin{align*}
D^\alpha (P (i\lambda)^{-n})&= \sum_{k=0}^{|\alpha|} n(n+1)\cdots (n+k-1) P_{k,\alpha}(\lambda) P(i\lambda)^{-n-k}\\
&=P(i\lambda)^{-n-|\alpha|}\sum_{k=0}^{|\alpha|} n(n+1)\cdots (n+k-1) P_{k,\alpha}(\lambda)P(i\lambda)^{|\alpha|-k},
\end{align*}
where the $P_{k,\alpha}$ are polynomials independent of $n$.
For all $\alpha$ occurring in the expansion one has $|\alpha|\leq 2M$, hence for each occurring $\alpha$
there exists a positive $C_\alpha$, such that
\begin{equation*}
|D^\alpha (P (i\lambda)^{-n})|\le C_\alpha n^{2M}   (|P(i\lambda_0)| -\varepsilon)^{-n-|\alpha|}
\end{equation*}
for all $n> 2M$, and all $\lambda \in \supp \psi$.
Therefore there exists a positive $C_1$ such that
\begin{equation*}
|(1- \Delta)^M \{ P(i\lambda)^{-n} \psi (\lambda)\}| \le C_1 n^{2M}   (|P(i\lambda_0)| -\varepsilon)^{-n}
\end{equation*}
for all $n>2M$, and all $\lambda \in \R^d$.
Using this in (\ref{intexpression}), we conclude that there exists a positive constant $C_2$,
such that
\begin{equation*}
|(1+|x|^2)^M \F \psi _n (x)| \le  C_2 n^{2M}  (|P(i\lambda_0)| -\varepsilon)^{-n},
\end{equation*}
for all $n> 2M$, and $x\in \R^d$.
It follows from this, that there exists a positive constant $C_3$, such that
\begin{equation*}\|\F \psi _n \|_q \le  C_3 n^{2M}   (|P(i\lambda_0)| -\varepsilon)^{-n},
\end{equation*}
for all $n> 2M$.
Then, since
\begin{equation*}
\langle \F f,\psi\rangle = \langle \F f,P(i\lambda)^n \psi_n \rangle =
\langle P(i\lambda) ^n\F f, \psi_n \rangle =
\langle\F ( P(\partial)^n  f), \psi_n \rangle=
\langle  P(\partial) ^n  f,\F \psi_n \rangle,
\end{equation*}
by (\ref{distfouriertransform}) and (\ref{intertwining}), we use the assumption that $P(\partial)^n f\in L^p(\R^d)$, for
all $n\in \N_0$, and H\"older's inequality, to obtain
\begin{equation*}
|\langle \F f,\psi\rangle| \le \| P(\partial) ^n f\|_p\|\F \psi_n \|_q
\le C_4 n^{2M}   (|P(i\lambda_0)| -\varepsilon)^{-n} \|( P(\partial) ) f\|_p,
\end{equation*}
for all $n>2M$. Since $|\langle \F f,\psi\rangle| >0$, we conclude that (\ref{mustshow2}) holds.
\end{proof}

Combining Proposition~\ref{prop:Lpkeyprop} above with Theorem~\ref{thm:realPWSchwartz} yields two results in an $L^p$-context.
In Section~\ref{sec:perspectives} we will comment on their possible interpretation in local spectral theory. The first result is
for Schwartz functions.

\begin{theorem}\label{thm:limSchwartz}
Let $P$ be a polynomial, $1\le p\le \infty$, and $f\in \S(\R^d)$.
Then in the extended positive real numbers
\begin{equation}
\lim _{n\to \infty}\| P(\partial )^n f\| _p ^{1/n} = R(P,\F f).
\end{equation}
\end{theorem}

\begin{proof}
In view of Proposition~\ref{prop:Lpkeyprop} it is sufficient to prove that in the extended positive real numbers
\begin{equation}
\limsup _{n\to \infty}\| P(\partial )^n f\| _p ^{1/n} \le R(P,\F f),
\end{equation}
and for this we may, and will, assume that $0< R(P,\F f)< \infty$.
We fix $M\in \N$, so that $(1+|x|)^{-M}\in L^p(\R^d)$.
From Theorem~\ref{thm:realPWSchwartz}, we see that there exists a positive constant $C$, such that
\begin{equation*}
\| (1+ |x|)^{M} P(\partial)^n f\| _\infty \le C n^M R(P,\F f)^n ,
\end{equation*}
for all $n\in \N$, which implies that
\begin{equation*}
\| P(\partial)^n f\| _p \le \| (1+ |x|)^{-M}\| _p\| (1+ |x|)^{N} P(\partial)^n f\| _\infty\le C^\prime n^M R(P,\F f)^n ,
\end{equation*}
for a new positive constant $C^\prime$. This completes the proof.
\end{proof}

The second consequence of the combination of Proposition~\ref{prop:Lpkeyprop} and Theorem~\ref{thm:realPWSchwartz} is valid for a
class of $L^p$-functions. There is overlap with the previous Theorem~\ref{thm:limSchwartz}, but neither statement contains the
other as a special case.

\begin{theorem}\label{thm:limLp}
Let $P$ be a polynomial and $1\le p\le \infty$. Suppose $P(\partial)^n f\in L^p(\R^d)$, for all $n\in \N_0$.
Assume furthermore that either $\F f$ has compact support, or that the set
$\{\lambda \in \R^d : |P(i\lambda)| \le R\}$ is compact for all $R\geq 0$.
Then in the extended
positive real numbers
\begin{equation} \label{characterize}
\lim_{n\to \infty}\| P(\partial )^n f\| _p ^{1/n} = R(P,\F f).
\end{equation}
\end{theorem}
\begin{proof}
From Proposition~{\ref{prop:Lpkeyprop}}, we have
\begin{equation}
\liminf_{n\to \infty}\| P(\partial )^n f\| _p ^{1/n} \ge R(P,\F f).
\end{equation}

As to the first case, assume that $\F f$ has compact support.
Fix $\varepsilon >0$. We will prove that
\begin{equation}\label{ineq}
\limsup_{n\to \infty}\| P(\partial )^n f\| _p ^{1/n} \le R(P,\F f)+\varepsilon.
\end{equation}
To this end, we choose a $\psi \in \S(\R^d)$, such that $\F \psi =1$ on $\supp \F f$, and $|P(i\lambda)| < R(P,\F f) +\varepsilon$,
for all $\lambda\in\supp \F \psi$. In particular, $R(P, \F \psi) < R(P,\F f) +\varepsilon$.
By the compactness of $\supp \F f$, this is indeed possible.
Since $\F (f* \psi) = \F f\cdot \F \psi = \F f$, we see that
$f*\psi = f$, hence
\begin{equation*}
\| P(\partial) ^n f\| _p = \| P(\partial)^n  (f* \psi)\|_p = \| f* P(\partial)^n   \psi\|_p \le \|f\|_p \| P (\partial)^n   \psi\|_1.
\end{equation*}
Then, as a consequence of Theorem~\ref{thm:limSchwartz} (or, given the proof of that result, actually already as a consequence
of Theorem~\ref{thm:realPWSchwartz}),
\begin{equation*}
\limsup_{n\to \infty} \| P(\partial) ^n f\| _p^{1/n}
\le \limsup_{n\to \infty} \| P (\partial)^n   \psi\|_1^{1/n}\le R(P,\F \psi) < R(P,\F f)+\varepsilon,
\end{equation*}
as desired.

As to the second case, assume that the set $\{\lambda \in \R^d : |P(i\lambda)| \le R\}$ is compact for all $R>0$.
If $R(P,\F f) < \infty$, then $\F f$ has compact support and the result follows from the first case.
If $R(P,\F f) = \infty$, the results follows from Proposition~\ref{prop:Lpkeyprop}.
\end{proof}

In particular, if one assumes that $\Delta ^n f\in L^p(\R^d)$, for all $n\in \N_0$, and $0\leq R <\infty$,
then $\lim_{n\to \infty}\| \Delta ^n f\| _p ^{1/n}\le R$, if, and only if, $\F f$ has support in the ball of radius $R$ around the origin.

For $p=2$, we can remove the compactness restrictions in Theorem~\ref{thm:limLp} by using the Plancherel theorem.
This is one of the two applications of the Plancherel theorem in the present paper, the other one
being in the proof of Theorem~\ref{prop:complexPWL2}.

\begin{theorem}\label{thm:limL2}
Let $P$ be a polynomial, and suppose $P(\partial)^n f\in L^2(\R^d)$, for all $n\in \N_0$. Then in the extended positive real numbers
\begin{equation*}
\lim_{n\to \infty}\| P(\partial )^n f\| _2 ^{1/n} = R(P,\F f).
\end{equation*}
\end{theorem}
\begin{proof}
By Proposition~{\ref{prop:Lpkeyprop}},
we only need to show that
\begin{equation}\label{L2ineq}
\limsup_{n\to \infty}\| P(\partial )^n f\| _2 ^{1/n} \le R(P,\F f).
\end{equation}
We may assume that $R(P,\F f) <\infty$.
For all $\psi \in \S(\R^d)$, we have
\begin{equation*}
\langle P(\partial )^n f,\psi\rangle=\langle P(\partial )^n f,\F \F^{-1}\psi\rangle=
\langle \F (P(\partial )^n f), \F^{-1}\psi\rangle=
\langle  P(i\lambda)^n \F f, \F^{-1}\psi\rangle,
\end{equation*}
so using the Plancherel theorem in the final step we obtain that
\begin{align*}
\left |\langle P(\partial )^n f,\psi\rangle\right |&= \left |\int _{\R ^d}
P(i\lambda)^n (\F f (\lambda))(  \F^{-1}\psi(\lambda))\, d\lambda \right |\\
&\le R(P,\F f)^n \int _{\R ^d}  |\F f (\lambda) \F^{-1}\psi(\lambda)|\,d\lambda\\
&\le  R(P,\F f)^n \| \F f\| _2 \| \F ^{-1} \psi \| _2\\
&= R(P,\F f)^n \|  f\| _2 \|  \psi \| _2.\\
\end{align*}
Since
\begin{equation*}
\| P(\partial)^n f \| _2 = \sup _{\psi \in \S(\R^d),\,\|\psi\|_2 = 1}
\left |\langle P(\partial )^n f,\psi\rangle\right |,
\end{equation*}
we conclude that
\begin{equation*}
\| P(\partial)^n f \| _2 \le R(P,\F f)^n\|  f\| _2,
\end{equation*}
and (\ref{L2ineq}) follows.
\end{proof}

If the Fourier transform $\F f$ of a function $f\in L^p(\R^d)$ is compact, then controlling $\Vert P(\partial)^n f\Vert_p$
for sufficiently many polynomials $P$ enables one to find the precise support of $\F f$.
This is expressed by part (c) and (d) of Theorem~\ref{thm:reconstructsupport} below, a key ingredient for which is the following.

\begin{lemma}\label{lem:SW}
Let $K$ be a non-empty compact subset of $\R^d$. Then
there exists a set $\mathcal P_K$ of polynomials such that
\[
K=\bigcap_{P\in\mathcal P_K}\{\lambda\in\R^d : |P(i\lambda)|\le\max_{\lambda\in K}|P(i\lambda)|\}.
\]
\end{lemma}

Indeed, as a consequence of the Stone-Weierstra{\ss} theorem, taking the set of all polynomials for $\mathcal P_K$ certainly works.
However, depending on the geometry of $K$ -- on which one may sometimes have a priori information -- much smaller sets may be sufficient.
For a ball, a single polynomial of degree two will do, and for convex hulls of finitely many points, finitely many polynomials of degree one
suffice. It is with this in mind that part (c) in Theorem~\ref{thm:reconstructsupport} and part (b)
in Theorem~\ref{thm:reconstructsupportcompdist} should be read.

\begin{theorem}\label{thm:reconstructsupport}
Let $1\le p\le \infty$, and assume the Fourier transform $\F f$ of $f\in L^p(\R^d)$ has compact support.
Let $K$ be a non-empty compact subset of $\R^d$. Then
\begin{enumerate}
\item[(a)] $P(\partial)f$ is in $L^p(\R^d)$, for all polynomials $P$;
\item[(b)] For any set $\mathcal P_K$ determining $K$ as in Lemma~\ref{lem:SW}, $\supp \F f\subset K$ if, and only if,
\[
\lim_{n\to\infty}\Vert P(\partial)^n f\Vert_p^{1/n}\leq\max_{\lambda\in K}|P(i\lambda)|,
\]
for all $P$ in $\mathcal P_K$.
\item [(c)] For any set $\mathcal P_{\supp \F f}$ determining $\supp \F f$ as in Lemma~\ref{lem:SW},
$\lambda\in\R^d$ is in $\supp\F f$ if, and only if,
\[
|P(i\lambda)|\leq \lim_{n\to\infty}\Vert P(\partial)^n f\Vert_p^{1/n},
\]
for all $P\in\mathcal{P}_{\supp\F f}$.
\item[(d)] One can reconstruct $\supp\F f$ as
\[
\supp \F f=\{\lambda\in\R^d : |P(i\lambda)|\leq \lim_{n\to\infty}\Vert P(\partial)^n f\Vert_p^{1/n},\,\textup{for all polynomials }P\}.
\]
\end{enumerate}
\end{theorem}

The existence of the finite limits in (b) and (c) is guaranteed by (a) and Theorem~\ref{thm:limLp}.

\begin{proof}
Choose $\psi\in C_c^\infty(\R^d)$ which is equal to 1 on $\supp \F f$. If $P$ is a polynomial, then, as distributions,
$P(\partial)f=f*P(\partial)\F^{-1}\psi$, as can be seen by taking Fourier transforms and using \cite[Theorem 30.4]{Tr}.
Since $L_1$ acts on $L_p$ by convolution, (a) follows.
Theorem~\ref{thm:limLp} gives (b) and (c). Since we already noted that the set of all polynomials
will always do for any compact set, (d) follows from (c).
\end{proof}

We now turn to general tempered distributions.

\begin{proposition}\label{prop:realPWtempdist1}
Let $T$ be a distribution of order $N$ with compact support, and let $P$ be a polynomial. Then, for each $R>R(P,T)$,
there exists a constant $C$, such that, for all $n\in\N$ and $x\in\R^d$,
\[
|P(\partial)^n \F^{-1}T(x)|\leq C n^N R^n (1+|x|)^N.
\]
\end{proposition}

\begin{proof}
Let $V_R=\{\lambda\in\R^d : |P(i\lambda)|\leq R\}$. Then $V_R$ contains the open neighborhood
$\{\lambda\in\R^d : |P(i\lambda)|< R\}$ of $\supp T$, hence we can choose $\psi\in C_c^\infty(\R^d)$ such
that $\psi=1$ on an open neighborhood of $\supp T$, and $\psi=0$ outside $V_R$. Since $T$ is of order $N$,
there exists a constant $C^\prime$, such that, for all $x_0\in\R^d$,
\begin{align*}
| (P(\partial)^n\F^{-1}T)(x_0)| &= |\F^{-1}(P(i\lambda)^n T)(x_0)|\\
&=|\F^{-1}(\psi P(i\lambda)^n T)(x_0)|\\
&=|\langle \psi P(i\lambda)^n T,e^{i x_0\cdot\lambda}\rangle|\\
&=|\langle T,\psi P(i\lambda)^n e^{i x_0\cdot\lambda}\rangle|\\
&\leq C^\prime \sum_{|\alpha|\leq N}\Vert D^\alpha (\psi P(i\lambda)^n e^{i x_0\cdot\lambda })\Vert_\infty.
\end{align*}

Induction with respect to $|\alpha|$ shows that
\[
D^\alpha (\psi P(i\lambda)^n e^{i x_0\cdot\lambda })=\sum_{k=0}^{|\alpha|} n(n-1)\cdots(n-k+1)\psi_{k,\alpha}(\lambda)Q_{k,\alpha}(x_0)e^{i x_0\cdot\lambda}P(i\lambda)^{n-k},
\]
where the $\psi_{k,\alpha}$ are smooth functions, independent of $n$ and $x_0$, and vanishing outside
$\supp \psi$, and the $Q_{k,\alpha}$ are polynomials, independent of $n$ and $\lambda$, and of degree at most $|\alpha|$.
For $n>N$, this can be written, for all $\lambda\in\R^d$, as
\[
P(i\lambda)^{n-N}\sum_{k=0}^{|\alpha|} n(n-1)\cdots(n-k+1)\psi_{k,\alpha}(\lambda)Q_{k,\alpha}(x_0)e^{i x_0\cdot\lambda}P(i\lambda)^{N-k}.
\]

The vanishing property of the $\psi_{k,\alpha}$ therefore implies that the supremum norm of this function is bounded by
\[
c_\alpha R^n n^{|\alpha|}(1+|x_0|)^{|\alpha|}
\]
for some constant $c_\alpha$. Hence there exists a constant $C$, such that the inequality in the theorem holds for $n>N$,
and increasing $C$ if necessary yields the desired inequality for all $n\in\N$.
\end{proof}

As to the reverse implication, we have the following result.

\begin{proposition}\label{prop:realPWtempdist2}
Let $f\in C^\infty(\R^d)$ be a tempered distribution, and suppose there exists a polynomial $P$, an integer $N\in\N_0$,
and constants $C,R\geq 0$, such that, for all $n\in\N$ and $x\in\R^d$,
\begin{equation}\label{eq:sufflargeequation}
|P(\partial)^n f(x)| \leq C n^N R^n (1+|x|)^N.
\end{equation}
Then $R(P,\F f)\leq R$.
\end{proposition}
Note that it is not assumed that $\supp\F f$ is compact.
\begin{proof}
Suppose $\lambda_0\in\R^d$ is fixed and such that $|P(i\lambda_0)|\geq R+\varepsilon$, for some $\varepsilon>0$.
Let $V=\{\lambda\in\R^d : |P(i\lambda)|> R + \varepsilon/2\}$, and suppose $\psi\in C_c^\infty(\R^d)$ is supported in $V$.
We will show that $\langle \F f,\psi\rangle=0$, hence $\lambda_0\notin\supp\F f$, which implies the theorem.

To this end we introduce, for $n\in\N$, the function $\psi_n(\lambda)=\psi(\lambda)P(i\lambda)^{-n}$, which is a well
defined compactly supported smooth function. If $M\in\N$ is any fixed integer, such that $(1+|x|^2)^{-M}\in L^1(\R^d)$,
then as in the proof of Proposition~\ref{prop:Lpkeyprop} one concludes that there is a constant $C^\prime\geq 0$, such
that, for all $x\in\R^d$ and all $n > 2M$,
\[
(1+|x|^2)^M |\F \psi_n(x)|\leq C^\prime n^{2M} \left(R+\frac{\varepsilon}{2}\right)^{-n}.
\]
Hence, if we choose $M\in\N$ such that $(1+|x|^2)^{-M}(1+|x|)^N\in L^1(\R^d)$, then, for $n>2M$ and $x\in\R^d$,
\begin{align*}
|\langle \F f,\psi\rangle|&=|\langle f,\F (P(i\lambda)^n\psi_n)\rangle|\\
&=|\langle f, P(-\partial)^n\F \psi_n\rangle|\\
&=|\langle P(\partial)^n f,\F\psi_n\rangle|\\
&=\left|\int_{\R^d} P(\partial)^n f (x) (1+|x|^2)^{-M} (1+|x|^2)^{M} \F \psi_n(x)\,dx \right|\\
&\leq\int_{\R^d} Cn^N R^n (1+|x|)^N (1+|x|^2)^{-M} C^\prime n^{2M}\left(R+\frac{\varepsilon}{2}\right)^{-n}\,dx\\
&\leq C^{\prime\prime}n^{N+2M}\left(\frac{R}{R+\frac{\varepsilon}{2}} \right)^n,
\end{align*}
and the theorem follows.
\end{proof}

In view of the proof of Theorem~\ref{prop:complexPWcompdist} below, we note that the proof of
Proposition~\ref{prop:realPWtempdist2} above shows that it is sufficient that \eqref{eq:sufflargeequation} holds for all but finitely many $n$.

Combining Proposition~\ref{prop:realPWtempdist1} and a special case of Proposition~\ref{prop:realPWtempdist2}
yields the following characterization of distributions with compact support, as a real counterpart of the classical complex result.

\begin{theorem}\label{thm:realPWcompdist}
Let $f\in C^\infty(\R^d)$ be a tempered distribution, and suppose  the set $\{\lambda\in\R^d : |P(i\lambda)|\leq R_0\}$ is compact for a
polynomial $P$ and a constant $R_0\geq 0$. Then the support of $\F f$ is contained
in $\{\lambda\in\R^d : |P(i\lambda)|\leq R_0\}$ if, and only if, for each $R>R_0$, there exist constants $N_R\in\N_0$ and $C_R\geq 0$, such that
\begin{equation}\label{pwrealcounterpartinequality}
|P(\partial)^n f(x)| \leq C_R n^{N_R} R^n (1+|x|)^{N_R},
\end{equation}
for all $n\in\N$ and $x\in\R^d$. If this is the case then, for all $R>R_0$, one can in fact take $N_R$ equal
to the order of $\F f$ in \eqref{pwrealcounterpartinequality}.
\end{theorem}

Theorems~\ref{thm:limLp} and~\ref{thm:reconstructsupport} for $L^p$ also have an analogue for distributions
with compact support, for which we need a definition.

\begin{definition}
Let $f\in C^\infty(\R^d)$ be a tempered distribution, such that $\supp\F f$ is compact. Let $P$ be a polynomial.
Then we define $\widetilde R(P,f)$ as the infimum of all $R\geq 0$, for which there exist constants $N\in\N_0$ and $C_{N,R}\geq 0$,
such that, for all $n\in\N$ and $x\in\R^d$,
\begin{equation*}
|P(\partial)^n f(x)| \leq C_{N,R} n^N R^n (1+|x|)^N.
\end{equation*}
\end{definition}

The analogue of Theorem~\ref{thm:limLp} is then the following, with  $\widetilde R(P,f)$ taking over the role of
$\lim_{n\to \infty}\| P(\partial )^n f\| _p ^{1/n}$.

\begin{theorem}
Let $f\in C^\infty(\R^d)$ be a tempered distribution, such that $\supp\F f$ is compact. Let $P$ be a polynomial.
Then $\widetilde R(P,f)=R(P,\F f)$.
\end{theorem}

\begin{proof}
By Proposition~\ref{prop:realPWtempdist1}, $\widetilde R(P,f)\leq R(P,\F f)$. The reverse inequality follows from
Proposition~\ref{prop:realPWtempdist2}.
\end{proof}

The reconstruction of supports, as in Theorem~\ref{thm:reconstructsupport} for the $L^p$-case, now takes the following form.

\begin{theorem}\label{thm:reconstructsupportcompdist}
Let $f\in C^\infty(\R^d)$ be a tempered distribution such that $\supp\F f$ is compact. Let $K$ be a non-empty compact subset of $\R^d$. Then
\begin{enumerate}
\item[(a)] For any set $\mathcal P_K$ determining $K$ as in Lemma~\ref{lem:SW}, $\supp \F f\subset K$ if, and only if,
\[
\widetilde R(P,f)\leq\max_{\lambda\in K}|P(i\lambda)|,
\]
for all $P$ in $\mathcal P_K$.
\item [(b)] For any set $\mathcal P_{\supp \F f}$ determining $\supp \F f$ as in Lemma~\ref{lem:SW}, $\lambda\in\R^d$ is in $\supp\F f$ if,
and only if,
\[
|P(i\lambda)|\leq \widetilde R(P,f),
\]
for all $P\in\mathcal{P}_{\supp\F f}$.
\item[(c)] One can reconstruct $\supp\F f$ as
\[
\supp \F f=\{\lambda\in\R^d : |P(i\lambda)|\leq  \widetilde R(P,f),\,\textup{ for all polynomials }P\}.
\]
\end{enumerate}
\end{theorem}

\begin{proof}
Part (a) follows from Proposition~\ref{prop:realPWtempdist1},
as does part (b). Part (c) is a special case of part (b).
\end{proof}

\section{Proving complex Paley--Wiener theorems without domain shifting}\label{sec:complexPW}

As mentioned in the introduction, there are instances where
proving a complex Paley--Wiener theorem for an integral
transform is not possible using domain shifting since the
integrand is not entire. In that case, results as in
Section~\ref{sec:realPW} may offer an alternative approach by
deriving the complex theorems from the real ones, as is done in
\cite{AdJ1}. To illustrate this point, we show how versions of a
number of classical complex Paley--Wiener theorems for the
Fourier transform follow from the real theorems in
Section~\ref{sec:realPW}. The strategy is to apply the Cauchy
formula to an entire function $f$, and then exploit the given usual
estimates for $f$ on $\C^d$ to obtain upper bounds for $\Vert
\partial_\xi^n f\Vert_\infty$ on $\R^d$, for all $n\in\N$ and
$\xi\in\R^d$. Combining this with a special case of
Proposition~\ref{prop:Lpkeyprop} gives upper bounds for
$\max_{x\in \supp \F f}|\xi\cdot x|$, for all $\xi\in \R^d$,
and then a basic separation theorem from convex analysis
implies the desired statement concerning $\supp \F f$ or $\supp \F^{-1} f$.

While this approach could offer an alternative where domain shifting is not valid, as yet it seems to be limited to supports
contained in a compact convex symmetric set. The symmetry condition is not necessary for the validity of the statement in the
Fourier case, but controlling $\Vert \partial_\xi^n f\Vert_\infty$ on $\R^d$ can never yield a result which is stronger than
that $\supp \F f$ is contained in some symmetric set, simply because $\max_{x\in A}|\xi\cdot x|=\max_{x\in -A}|\xi\cdot x|$.
As demonstrated by, e.g., Theorem~\ref{thm:reconstructsupport}, it is in principle possible to reconstruct $\supp \F f$ by
controlling $\Vert P(\partial)^n f\Vert_\infty$ on $\R^d$ for sufficiently many polynomials $P$, so that one might hope to be
able to infer the general complex case via the real Proposition~\ref{prop:Lpkeyprop} by invoking more polynomials than just the
first degree homogeneous ones. The combinatorics that arise when trying to obtain estimates via the Cauchy formula for
$\Vert P(\partial)^n f\Vert_\infty$ for more general $P$ seem to get rather involved though, and it remains to be seen whether
a real approach to the full complex result is actually feasible.

For functions, the only result we will need from Section~\ref{sec:realPW} is Proposition~\ref{prop:Lpkeyprop} for $p=\infty$
and homogeneous polynomials of degree one. If the function in Proposition~\ref{prop:Lpkeyprop} is a Schwartz function, then
we can also use the following ad hoc result with a much simpler proof. We include it to illustrate that deriving the complex
Theorem~\ref{prop:complexPWSchwartz} via real results is rather elementary.

\begin{lemma}\label{lem:liminfSchwartz}
Let $P$ be a homogeneous polynomial with real coefficients, $f\in \S(\R^d)$, and $1\leq p\leq\infty$.
Then in the extended positive real numbers
\begin{equation} \label{mustshow1Schwartz}
\liminf _{n\to \infty}\| P(\partial )^n f\| _p ^{1/n} \ge R(P,\F f).
\end{equation}
\end{lemma}
Obviously the result remains true for polynomials which are a scalar multiple of a polynomial as in the lemma.
\begin{proof}
Let $\lambda _0\in \supp \F f$, and assume $P(i\lambda _0) \ne 0$.
We choose and fix an $\varepsilon >0$, such that
$0 < \varepsilon < |P(i\lambda _0)|$.
We will show that
\begin{equation}\label{mustshow2Schwartz}
\liminf _{n\to \infty}\| P(\partial )^n f\| _p ^{1/n} \ge |P(i\lambda _0)| - \varepsilon,
\end{equation}
which will establish (\ref{mustshow1Schwartz}).
Define, for $j\in \{0,1,2,3\}$,
\begin{equation*}
\psi_j (\lambda) = \overline{\F (P(\partial)^jf)(\lambda) }\qquad (\lambda \in \R^d).
\end{equation*}
Then, using (\ref{distfouriertransform}) and (\ref{intertwining}),
H\"older's inequality yields
\begin{align*}
\left \Vert  P(\partial)^{4n+j} f \right \Vert _{p}\left \Vert \F \psi_j \right \Vert _{q}
&\ge|\langle P(\partial)^{4n+j} f, \F \psi_j \rangle  |\\
&=|\langle \F (P(\partial)^{4n+j}  f),\psi_j\rangle  |\\
&=\left |\int _{\R^d} P(i\lambda)^{4n} \F(P(\partial)^{j}  f)(\lambda) \psi_j (\lambda)\, d\lambda \right |\\
&=\int _{\R^d} |P(i\lambda)|^{4n} |\F (P(\partial)^{j}  f)|^2\, d\lambda \\
&\ge (|P(i\lambda _0)|-\varepsilon)^{4n}
\int _{\{\lambda : |P(i\lambda)|\ge |P(i\lambda _0)| -\varepsilon\}}|P(i\lambda)^{2j}|
| \F f(\lambda)|^2\,d\lambda.
\end{align*}
As the integral in the last line is strictly positive due to the fact that $\lambda_0\in\supp \F f$, this yields (\ref{mustshow2Schwartz}).
Note that both assumptions on $P$ were used in passing from the third line to the fourth in the above display of equations.
\end{proof}

Before passing to the complex Paley--Wiener theorems, let us establish terminology. If $A$ is a non-empty subset of $\R^d$, then we define
the supporting function $H_A:\R^d\to\R^d$ of $A$ as, $H_A (x) = \max _{a\in A}  a\cdot x$, for $x\in \R^d$.
\begin{definition}Let $A$ be a subset of $\R^d$, and let $f:\C^d\to\C$.
We say that $f$ is an entire function on $\C^d$ of exponential type corresponding to
$A$, if $f$ is entire, and if there exists a positive constant $C$, such that,
\begin{equation}\label{exptype}
|f(z) | \le C e^{H_A (\Im z)}\qquad (z\in \C^d).
\end{equation}
We further say that $f$ is a rapidly decreasing
entire function on $\C^d$ of exponential type corresponding to $A$, if $f$ is entire, and, for each $n\in \N_0$,
there exists a positive constants $C_n$, such that
\begin{equation}\label{rapidexptype}
|f(z) | \le C_n (1+|z|)^{-n} e^{H_A (\Im z)}\qquad (z\in \C^d).
\end{equation}
\end{definition}

As is well known, if an entire function $f$ satisfies (\ref{rapidexptype}), then, using
the Cauchy formula, one sees that the partial derivatives
$\partial _\xi^n f$ satisfy similar estimates, for all $\xi\in \R^d,\,n\in \N$. Hence
the restriction of $f$ to $\R^d$ is a Schwartz function,
explaining the terminology.

The following result is the pivot in the transition from the complex to the real domain in the case of functions.

\begin{lemma}\label{lem:expestimates}
Let $A$ be a non-empty compact symmetric subset of $\R^d$, and suppose $f$ is an entire function on $\C^d$ of exponential type corresponding to $A$.
Then, for all $\xi\in \R^d,\,n\in \N$, the partial derivatives $\partial _\xi ^n f$ are bounded on $\R^d$, and
\begin{equation}\label{stirling}
\limsup _{n\to\infty}\|\partial _\xi^n f \|_\infty ^{1/n} \le H_A (\xi).
\end{equation}
\end{lemma}
\begin{proof}
We claim that
\begin{equation}\label{estimate}
\|\partial _\xi^n f \|_\infty  \le C \frac {n!e^n}{n^n} ({H_A (\xi)})^n,
\end{equation}
for $\xi\in \R^d,\,n\in \N$,
where the constant $C$ is as in (\ref{exptype}). Indeed, Cauchy's theorem yields, for arbitrary $r>0$, that
\begin{equation*}
\partial _\xi^n f (x)=  \frac{d^n}{dt^n}\vert_{t=0} \{ t\mapsto f(x+t\xi)\}=\frac {n!}{2\pi i}\oint_{|t| = r}
\frac {f(x+t\xi)}{t^{n+1}}\, dt \qquad (x\in \R^d),
\end{equation*}
and thus
\begin{equation*}
|\partial _\xi^n f (x)|\le C \frac {n!}{r^n} e^{\max_{a\in A,|t|=r}a\cdot\Im(t\xi)}=
C \frac {n!}{r^n} e^{r\max_{a\in A}|a\cdot\xi|}=C \frac {n!}{r^n} e^{rH_A(\xi)},
\end{equation*}
where the symmetry of $A$ is used in the final equality.
If $H_A (\xi) =0$, letting $r\to \infty$ shows that $\partial _\xi^n f = 0,\,n\in\N$,
and(\ref{estimate}) is proved. If $H_A (\xi) \ne 0$, then $H_A(\xi)>0$ by the symmetry of $A$, hence we can choose $r = n/H_A (\xi)$,
and this again establishes (\ref{estimate}).
Then (\ref{stirling}) follows from Stirling's formula.
\end{proof}

We now use Lemmas~\ref{lem:liminfSchwartz}
and~\ref{lem:expestimates} to prove the symmetric version of
the complex Paley--Wiener theorem for smooth functions.

\begin{theorem}\label{prop:complexPWSchwartz}
Let $A$ be a non-empty compact, convex and symmetric subset of $\R^d$, and suppose $f:\R^d\to\C$. Then the following are equivalent:
\begin{enumerate}
\item[(a)] $f$ is the Fourier transform of a smooth function with support contained in $A$;
\item[(b)] $f$ extends to a rapidly decreasing entire function on $\C^d$ of exponential type corresponding to $A$.
\end{enumerate}
\end{theorem}
\begin{proof}
It is classical, and easy to see, that (a) implies (b).

As to the converse, assume that $f$ is a rapidly decreasing entire function on $\C^d$ of exponential type corresponding to $A$.
Suppose that $\lambda_0\in \supp \F f$, but $\lambda_0\not\in A$. Then, since $A$ is convex and closed, and $\{\lambda_0\}$ is
convex and compact, by a standard separation result there exists $\xi _0\in \R^d$,
such that $\lambda_0\cdot \xi_0> H_A(\xi_0)$. Consider the polynomial $P_{\xi_0}(\lambda) =  \lambda\cdot \xi_0$, for $\lambda\in \R^d$.
By Lemma~\ref{lem:liminfSchwartz} and Lemma~\ref{lem:expestimates},
\begin{align*}
\sup _{\lambda \in \supp \F f} |\lambda\cdot\xi_0| &= \sup _{\lambda \in \supp \F f} |P_{\xi_0}(i\lambda)|\le
\liminf _{n\to \infty}\| \partial _{\xi_0} ^n f\| _\infty ^{1/n}\\
&\le
\limsup _{n\to\infty}\|\partial _{\xi_0} ^n f \|_\infty ^{1/n} \le H_A (\xi_0),
\end{align*}
so that $H_A(\xi_0)<\lambda_0\cdot\xi_0\leq |\lambda_0\cdot\xi_0|\leq H_A(\xi_0)$, a contradiction.
We conclude that $\F f$ has support in $A$.
Since $\F ^{-1}f(x)= \F f(-x)$, for $x\in\R^d$, the symmetry of $A$ implies that $\F^{-1}f$ also has support contained in $A$.
\end{proof}

The complex symmetric Paley--Wiener theorem for $L^2$-functions
follows in the same fashion, using Lemma~\ref{lem:expestimates}
and Proposition~\ref{prop:Lpkeyprop}; the latter as a
stronger replacement of Lemma~\ref{lem:liminfSchwartz} in the
proof of Theorem~\ref{prop:complexPWSchwartz} above.
The proof contains one of the two applications of the
Plancherel theorem in the present paper, the other one being in
the proof of Theorem~\ref{thm:limL2}.

\begin{theorem}\label{prop:complexPWL2}
Let $A$ be a non-empty compact, convex and symmetric subset of
$\R^d$, and suppose $f$ is a measurable function on $\R^d$
representing a tempered distribution. Then the following are
equivalent:
\begin{enumerate}
\item[(a)] $f$ is the Fourier transform of a function in $L^2(\R^d)$ with support contained in $A$;
\item[(b)] $f$ is in $L^2(\R^d)$ and extends to an entire function on $\C^d$ of exponential type corresponding to $A$.
\end{enumerate}
\end{theorem}
\begin{proof}
It is obvious that (a) implies (b).
Assuming (b), the Plancherel theorem yields that $\F^{-1}f\in L^2(\R^d)$.
Applying Lemma~\ref{lem:expestimates}, and the case $p=\infty$ of Proposition~\ref{prop:Lpkeyprop}, to $f$ yields that, for all $\xi\in\R^d$, $\sup_{x\in\supp\F f}|\xi\cdot x|\leq H_A(\xi)$.
A separation argument as in the proof of Theorem~\ref{prop:complexPWSchwartz} shows that $\supp \F f$ is contained in $A$, and then the symmetry of $A$ implies that the same is true for $\supp \F^{-1} f$.
\end{proof}

\begin{remark}
Let $A$ be a non-empty compact, convex and symmetric subset of
$\R^d$. Let $1\le p\le \infty$, and define $\H^{p}_{A} (\C ^d)$
as the space of entire functions $f$ of exponential type
corresponding to $A$, whose restriction to $\R^d$
belongs to $L^{p}(\R^d)$. As in the proof of
Theorem~\ref{prop:complexPWL2} above, applying
Lemma~\ref{lem:expestimates}, and the case $p=\infty$ of
Proposition~\ref{prop:Lpkeyprop}, to such $f$ yields that the
distribution $\F f$ (and hence also $\F ^{-1}f$) has support in
$A$, for all $f\in\H^{p}_{A} (\C^d)$. In particular, for $1\le p\le
2$ with conjugate exponent $q$, we have established by real
methods that $\F$ and $\F ^{-1}$ map $\H^{p}_{A} (\C^d)$ into a
subspace of $L^q(\R^d)$ consisting of functions with
distributional support in $A$.
\end{remark}

The symmetric complex Paley--Wiener theorem for distributions
with compact support can be derived using the real Proposition~
\ref{prop:realPWtempdist2}.

\begin{theorem}\label{prop:complexPWcompdist}
Let $A$ be a non-empty compact, convex and symmetric subset of $\R^d$, and suppose $f:\R^d\to\C$. Then the following are equivalent:
\begin{enumerate}
\item[(a)] $f$ is the Fourier transform of a distribution with support contained in $A$;
\item[(b)] $f$ extends to a entire function on $\C^d$, and
    there exists an integer $N\in\N_0$, and a constant $C$, such that
    $|f(z)|\leq C (1+|z|)^N e^{H_A(\Im z)}$, for all
    $z\in\C^d$.
\end{enumerate}
If this is the case, then the integer $N$ in \textup{(b)} can be taken to be the order of $\F^{-1}f$.
\end{theorem}

\begin{proof}
We need only prove that (b) implies (a), the rest being
obvious. Fix $\xi\in\R^d$. We claim that there exists a
constant $C^\prime$ such that, for all $x\in\R^d$ and $n>N$,
\begin{equation}\label{eq:claimedinequality}
|\partial_\xi^n f(x)|\leq C^\prime n^{N+1}H_A(\xi)^n(1+|x|)^{N+1}.
\end{equation}
One this has been established, we infer from
Proposition~\ref{prop:realPWtempdist2} and the subsequent remark,
that, for all $\xi\in\R^d$, $\max_{\lambda\in\supp \F^{-1}
f}|\lambda\cdot\xi|\leq H_A(\xi)$. Statement (a) then follows
from a separation argument as in the proof of
Theorem~\ref{prop:complexPWSchwartz}. As to
\eqref{eq:claimedinequality}, a contour integral as in the
proof of Lemma~\ref{lem:expestimates} yields that, for all
$x\in\R^d$, $n\in\N$, and $r>0$,
\[
|\partial_\xi^n f(x)|\leq C \frac{n!}{r^n}(1+|x|+r|\xi|)^N e^{rH_A(\xi)},
\]
where $C$ is as in (b). If $H_A(\xi)=0$, then, as $n>N$,
letting $r\to\infty$ establishes \eqref{eq:claimedinequality}.
If $H_A(\xi)\neq 0$, taking $r = n/H_A (\xi)>0$ shows that, for
all $x\in\R^d$ and $n\in\N$,
\begin{align*}
|\partial_\xi^n f(x)|&\leq  C \frac{n!e^n}{n^n}H_A(\xi)^n\left(1+|x|+\frac{n|\xi|}{H_A(\xi)}\right)^N\\
&\leq C n^N \frac{n!e^n}{n^n}H_A(\xi)^n\left(1+|x|+\frac{|\xi|}{H_A(\xi)}\right)^N,
\end{align*}
and then Stirling's formula yields \eqref{eq:claimedinequality}.

\end{proof}

\section{Literature review}\label{sec:comparison}

In Section~\ref{sec:realPW}, we have established real
Paley--Wiener theorems for the Fourier transform, and in
Section~\ref{sec:complexPW} these have been used to derive complex
Paley--Wiener theorems without domain shifting. We will now
describe the historical development of the field and compare our results with the literature.

In the late fifties and early sixties, the Russian school studied families of test function spaces on the real line
satisfying certain duality relations under the Fourier transform,
the so-called Gelfand--Shilov spaces, see \cite[Chapter~4]{GeSchi}.
For example, the Gelfand--Shilov spaces $S_\alpha^\beta$ $(\alpha, \,\beta \ge 0)$ of type $S$ consist of all smooth functions
$f$ on $\R$ satisfying
\begin{equation}\label{gelfand--shilov}
| x^k f^{(q)}(x)| \le C A^k B^q k^{k\alpha} q^{q \beta}\quad(x\in \R, k\in \N_0, q \in \N_0),
\end{equation}
for constants $A,B$ and $C$ depending on the function $f$. We note that $\F (S_\alpha^\beta) = S^\alpha_\beta$, and that
$S_0^\beta$ consists of functions with compact support.
The growth estimates (\ref{gelfand--shilov}) are similar to the ones in Theorem~\ref{thm:realPWSchwartz}(b),
but are concerned with repeated application of two operators rather than one as in our case, and the differential operator is always monomial.
The $L^p$-norms for $1\leq p<\infty$ are furthermore not considered in \cite{GeSchi}.
In spite of the difference with later work by other authors, some of the results on the Gelfand--Shilov spaces in \cite{GeSchi}
should probably be thought of as the first real Paley--Wiener type theorems.

In 1990, see \cite{Ba1}, Bang studied the growth of derivatives of an
$L^p$-function $f$ on the real line, and, assuming that all derivatives of $f$ are also $L^p$-functions,
found the relation, \cite[Theorem~1]{Ba1},
\begin{equation}\label{bangfirst}
 \lim_{n\to \infty}\|  f^{(n)}\| _{p} ^{1/n} = \sup \{|\lambda| \,:\, \lambda \in \supp \F f\} ,
\end{equation}
between the support of the Fourier transform of $f$ and the growth of these derivatives.
This is our Theorem~\ref{thm:limLp} for $d=1$ and $P(x) = x$. We believe that \eqref{bangfirst} is the first result
relating the $L^p$-norm of a repeated application of a differential operator to $f$
to the supremum of the modulus of the corresponding multiplier
on the support of the transform of $f$, i.e.,\,the
subject proper of the present article.
Bang's proof of (\ref{bangfirst}) uses the Kolmogorov and Bernstein inequalities and the complex Paley--Wiener theorem.
There is a simplified approach in \cite{An1}.

Bang, on his own and with coauthors, has since generalized \eqref{bangfirst} to
the Fourier transform in Euclidean space in higher dimensions, as well as to Orlicz spaces and Lorentz spaces
(with the appropriate associated operators), see \cite{Ba12}, \cite{Ba2},
\cite{BaY}, \cite{Ba1995b}, \cite{Bang}, \cite{Ba21}, \cite{Ba22}, \cite{Ba3}, \cite{Ba4}, \cite{Ba5},
\cite{Ba6}, \cite{BMa}, \cite{BMo1} and \cite{BMo2}.
In \cite{Ba1}, Bang also covers the case of $L^p$-functions on the one-dimensional torus for $P(x)=x$,
again using a complex Paley--Wiener theorem.

An approach using the Plancherel formula and intertwining properties was suggested by Tuan in 1995, see \cite{Tu0} and \cite{Tu1}.
The philosophical approach to more general set-ups is given by
\cite{Tu00} and \cite{Tu2}, in particular \cite[Theorem 1]{Tu2}, which explains, also in higher
dimensions, why the existence of such generalizations is
plausible for $p=2$ in the presence of a Plancherel theorem, at
least for functions in a suitable subspace of the $L^2$-space
under consideration.
The new approach made it possible for Tuan and various coworkers to generalize Bang's result to several other integral transforms,
mostly defined via differential equations (Dirac, Sturm--Liouville,...) on the real line, see
\cite{Al1}, \cite{Al2}, \cite{AlT1}, \cite{AlT2},
\cite{Tu01}, \cite{Tu02}, \cite{Tu03}, \cite{Tu04}, \cite{Tu3}, \cite{Tu4}, \cite{TuIS1}, \cite{TuIS2},
\cite{TuZa0}, \cite{TuZa1} and \cite{TuZa}.

From 2000 onwards, the methods by Bang and Tuan, and the simplified approach from \cite{An1}, have been used to
find real Paley--Wiener type theorems for a number of integral transform. We refer to \cite{Ab},
\cite{An-1}, \cite{An0}, \cite{An01}, \cite{An2}
\cite{An3}, \cite{An4}, \cite{AdJ1},
\cite{BBM}, \cite{CT}, \cite{COT}, \cite{JT}, \cite{MT}, \cite{OT},
\cite{Th} for more details. Finally, we mention results by
Pesenson in a similar vein, see \cite{Pes1}, \cite{Pes2}, \cite{Pes3}, \cite{Pes4}, \cite{Pes5}.
Here, in an $L^2$-context, bandlimited functions, i.e., functions with a compactly supported Fourier transform,
on various manifolds are described by generalized Bernstein inequalities, and applications
to sampling theory  are given.

We will now discuss the most important contributions by Bang and Tuan, and compare them with our results.

After the initial work in one dimension in \cite{Ba1}, Bang considered $L^p$-functions in higher dimensions
in \cite{Ba2}, but still for monomial operators only. The
main result is the following: Let $0<p\le \infty$, $f\in
L^p(\R^d)$, and assume that $\supp \F f$ is compact. Then
\begin{equation}\label{bang2}
\lim _{|\alpha| \to \infty} \left (\|\partial ^\alpha f\|_p/
\sup _{\lambda \in \supp \F f}|\lambda^\alpha|\right )^{1/|\alpha|} =1,
\end{equation}
where $\alpha$ is a multi-index, and $\lambda^\alpha$ has the
usual meaning. The proof is rather technical, and uses, e.g.,
Sobolev theory for elliptic equations. It is further shown in
\cite{Ba1995b}, that one can let $p$ depend on $\alpha$ by
choosing $1\leq p_\alpha\leq\infty$, and still obtain similar
results.

The step to arbitrary polynomial operators and also to
distributions was taken by Bang in \cite{Ba3}, which
is not concerned with $L^p$-norms, but with pointwise
estimates. Some of the proofs are again rather technical, using
Sobolev theory for elliptic equations in combination with
structure theory for distributions. One of the main results,
\cite[Theorem 1]{Ba3}, which is established using these
techniques is as follows: Let $K$ be an arbitrary compact set
in $\R^d$. Then a tempered distribution $T$ has support in $K$
if, and only if, $\F T$ is a smooth function with the property
that there exists an integer $N\in\N_0$, such that, for every $\delta >0$,
there exists a constant $C_{\delta}$, such that
\begin{equation}\label{bang3}
|P (i\partial) \F T (\lambda)| \le C_{\delta} (1+|\lambda|)^{N}\sup_{z \in K_{(\delta)}}|P(z)|
\qquad (\lambda\in \R^d),
\end{equation}
for all polynomials $P$, where $K_{(\delta)}$ is defined as the
$\delta$-neighborhood of $K$ in $\C^d$. There is a similar
result, \cite[Theorem 2]{Ba3}, for Schwartz functions. While
these two results certainly yield information about the
support they are of a different flavor than ours, since $\C^d$ is still involved and our results are entirely formulated on $\R^d$.
Furthermore, such information can only be obtained once one
controls the left hand side for all polynomials, with a
universal constant on the right hand side. This may not always
be feasible, and our single polynomial results, giving also
non-compact information, can then be more practical.

The
results in \cite{Ba3} which are closest to ours, are Theorem~5
for compactly supported distributions, and Theorem~6 for
Schwartz functions. Theorem~5 follows from our
Propositions~\ref{prop:realPWtempdist1}
and~\ref{prop:realPWtempdist2}, but its proof in \cite{Ba3} is
much more involved. Conversely, Bang's Theorem~5 implies
our Proposition~\ref{prop:realPWtempdist1}, but not our
Proposition~\ref{prop:realPWtempdist2}, which has no compactness
assumptions. Bang's Theorem~6 is equivalent to our
Theorem~\ref{thm:realPWSchwartz}, and has a proof in the same
spirit as our proofs.\footnote{There appears to be a slight
mistake in the formulation of \cite[Theorem~6]{Ba3}. In our
opinion a factor $m^N$ should be added to the right hand side
of (3.13) in \cite{Ba3}, corresponding to the factor $n^N$ in
our Theorem~\ref{thm:realPWSchwartz}. This factor takes into
account that the number of terms in the second summation on
page~28 of \cite{Ba3} depends on $m$. There are $m^{|\nu|}$
terms and this needs to be estimated from above by $m^N$.} Bang
was also the first to give reconstructing theorems for the
support of distributions and functions in special cases. He has
results in two contexts: using monomials \cite[Theorems~3,
4]{Ba3}, and using polynomials of a given maximal degree
\cite[Theorems~7, 8]{Ba3}. Our reconstruction results,
Theorem~\ref{thm:reconstructsupport} for functions and
Theorem~\ref{thm:reconstructsupportcompdist} for distributions,
have simpler proofs and are more general, as they are
applicable to arbitrary sets of polynomials.

A further step was taken by Tuan in \cite{Tu1}, where
he studies $L^p$-functions in arbitrary dimension, but also
includes repeated application of some non-monomial operators, a
case not considered previously by Bang. Most of the results are for
$p=2$, as, e.g., Theorems~1 and~3 in \cite{Tu1} for the
operators $\Delta$ and $\partial_\xi$, respectively. For
general $p$, he obtains \cite[Theorem 4]{Tu1}, which reads as
follows: Let $P$ be a non-constant polynomial; then the Fourier
transform $\F f$ of $f\in \S(\R^d)$ vanishes outside $\{x\in
\R^d : |P(x)|\le 1\}$ if, and only if,
\[
\limsup _{n\to \infty} \| P(i\partial )^n f\| _p ^{1/n} \le 1,\quad 1\le p\le \infty.
\]
This is a consequence of the stronger
Theorem~\ref{thm:limSchwartz} in the present paper. The proof
in \cite{Tu1} uses the Hausdorff--Young inequality, the
Plancherel theorem and H\"older's inequality, and it is not
obvious how this method could adapted to yield our
Theorem~\ref{thm:limLp} for non-smooth $L^p$-functions, for
which we have, in fact, not found any previously occurring
related result for non-monomial operators. For comparison we
also remark that, although the statement itself is correct for
arbitrary polynomials, in our opinion the proof of
\cite[Theorem 4]{Tu1} only establishes the result for
polynomials which are real valued -- it is not stated in
\cite{Tu1} which polynomials are considered. The reason is
that, when passing from equation (37) to (38) in \cite{Tu1}, it
is used that
$P^k(-i\partial)\overline{f}=\overline{P^k(i\partial)f}$
for all $k$ and $f$, and for this one needs $P$ to be real
valued. The same remark applies to the transition from equation
(29) to (30) in the proof of \cite[Theorem 2]{Tu2}, while our
Theorem~\ref{thm:limSchwartz} is stated for
arbitrary non-constant polynomials: in our opinion, the proof
in \cite{Tu2} only establishes \cite[Theorem~2]{Tu2} for real
valued polynomials. This restriction originates from the
application of the Plancherel theorem in both proofs, whereas
our proofs of the main results in the present paper, such as
Theorem~\ref{thm:limSchwartz} for arbitrary polynomials, are
all based on the inversion theorem. This different approach
avoids this limitation, which we think to be inherent to proofs
based on a Plancherel theorem.

The results and approach mentioned above were in \cite{Tu3}
and \cite{MT} generalized to the Dunkl transform,
where we recall that the Dunkl transform, \cite{Du}, \cite{dJ1},
is a deformation of the Fourier transform that
specializes to the Fourier transform if all the so-called root
multiplicities are zero. This also means that we believe the results in \cite{Tu3}
and \cite{MT} are valid only for real valued polynomials.

Theorem~3 in \cite{Tu2} is
concerned with repeated application of $e^\Delta$ to Schwartz
functions in arbitrary dimension in an $L^p$-context. We have no analogue of this
result which, when interpreted, in our opinion gives additional evidence that
some real Paley--Wiener theorems express the validity of a local
spectral radius formula, a point we will argue in Section~\ref{sec:perspectives}. In the one-dimensional case
and for Schwartz functions in an $L^p$-context,
\cite[Theorem~4]{Tu2} involves repeated application of another
operator not of the form $P(\partial)$. There is also a result
for the asymptotic action of the heat semigroup on $L^2(\R^d)$ available, see \cite[Theorem~2]{Tu1}.

As to the complex theorems, apart from the results in
Section~\ref{sec:complexPW}, we are not aware of derivations of
complex Paley--Wiener theorems from real ones, other than the
results for the one-dimensional Dunkl transform in \cite{AdJ1}.
The idea seems to be still very young and needs further
testing. A context for a possible application is the following:
In higher dimension, the results for complex Paley--Wiener
theorems for the Dunkl transform in \cite{dJ2} and \cite{ThaXu}
seem to be not yet optimal, as the version for supports in
compact convex invariant sets is thought to be true in general,
but has only been proved in special cases \cite{dJ2}. Even
though -- arguing analogously to the remarks in
Section~\ref{sec:complexPW} for the Fourier transform -- an
approach through real methods can only be expected to give a
complex theorem for supports in a symmetric compact convex
invariant set, such a version of a general complex theorem
would already be a considerable improvement. With the necessary
real Paley--Wiener theorem already established (the case of
homogenous polynomials of degree one and $p=\infty$ of
\cite[Theorem~4.1]{MT}), the future may learn whether this can
actually be achieved in this fashion, analogously to the
one-dimensional case in \cite{AdJ1}. Similarly, the
applicability of this relatively new approach to complex
theorems could be investigated for transforms as in \cite{An0}
and \cite{TuZa}, where some real theorems are also already
available.

\section{Perspectives and connection with local spectral theory}\label{sec:perspectives}

In this section we mention some possible future developments, one of which is the interpretation of
some real Paley--Wiener theorems as local spectral radius formulas.

Firstly, as was already mentioned in
Sections~\ref{sec:intro}, \ref{sec:complexPW}, and
\ref{sec:comparison}, we hope that the idea of deriving complex
Paley--Wiener theorems from real ones will prove to be a useful
one in other contexts as well, especially if the usual domain
shifting is not applicable. This is indeed the case for the
one-dimensional Dunkl transform, see \cite{AdJ1}, and the higher
dimensional situation can be a new testing ground.

Secondly, there are potential applications in support theorems
for partial differential equations with constant coefficients.
For example, if $f$ and $g$ are compactly supported smooth
functions, such that $P(\partial)f=g$, then the convex hull of
the support of $f$ equals that of $g$ as follows, e.g., from
the theorem of supports \cite[Theorem~4.3.3]{Hoer}. This seems
to be the best result available. However,
Theorem~\ref{thm:limSchwartz} gives an opportunity to obtain
extra information by choosing a polynomial $Q$ and estimating
\[
\lim_{n\to\infty} \Vert Q(\partial)^n \F f(\lambda)\Vert_p^{1/n}=\lim_{n\to\infty} \left\Vert Q(\partial)^n\frac{\F g(\lambda)}{P(i\lambda)}\right\Vert_p^{1/n}
\]
from above, for some $1\leq p\leq\infty$. If this limit is
known to be at most $M$, then the support of $f$ must be
contained in $\{x\in\R^d : |Q(-ix)|\leq M\}$. Carrying this out
in a concrete situation requires a judicious choice for $Q$,
probably $p=\infty$, and a skillful handling of the
combinatorics. The behavior of $P(i\lambda)$ near its zero
locus will also enter the computation, but for this general
lower bounds are available in \cite[Lemma~5.7]{TrevesLPDE}.
Notwithstanding the complications that will arise when pursuing this
idea we mention it nevertheless, since in principle it can
yield information which to our knowledge can not be obtained
otherwise.

Thirdly, as already mentioned in Section~\ref{sec:comparison},
Bang also obtained an elementary version of a real
Paley--Wiener theorem for the one-dimensional torus in
\cite{Ba1}. This has not been developed further, but it seems
very likely that the methods of proof in the present paper can be adapted to establish real Paley--Wiener theorems for
$\R^{d_1}\times\T^{d_2}$. Is it also possible to do this for Lie groups other than the connected abelian groups,
with the operators $P(\partial)$ replaced by the operators in
the center of the universal enveloping algebra? The answer is at least affirmative for compact connected groups \cite{AdJ2}.

Fourthly, we think that the relation between real Paley--Wiener
theory and local spectral theory deserves further attention, in particular the interpretation of
Theorems~\ref{thm:limSchwartz} and~\ref{thm:limLp} as
statements belonging to that field. We will now elaborate on
this, starting with bounded operators on a Banach space, where
the results of our interest are already available in general, see, e.g., \cite{LaurNeu}, from which the
material below on the bounded case is mostly taken.

If $X$ is a Banach space, $T$ a bounded operator on $X$, and
$x\in X$, then $z_0\in\C$ is said to be in the local resolvent
set of $T$ at $x$, denoted by $\rho_T(x)$, if there is an open neighborhood $U$ of $z_0$
in $\C$, and an analytic function $\phi: U\to X$, sending $z$ to
$\phi_z$, such that
\begin{equation}\label{eq:localresolvent}
(T-z)\phi_z=x\qquad(z\in U).
\end{equation}
The local spectrum $\sigma_T(x)$ of $T$ at $x$ is the
complement of $\rho_T(x)$ in $\C$. Clearly
$\sigma_T(x)\subset\sigma (T)$, where $\sigma (T)$ is the
spectrum of $T$. The operator $T$ is said to have the
single-valued extension property (SVEP) if, for every
non-empty open set $U\subset\C$, the only analytic solution $\phi:U\to X$ of the
equation $(T-z)\phi_z=0$ $(z\in U)$ is the zero solution. This is equivalent to requiring that the
analytic local function $\phi$ in \eqref{eq:localresolvent}
is uniquely determined, so that one can speak of ``the'' local
resolvent function $\phi$. For an operator $T$, having SVEP is also equivalent to
0 being the only element in $X$ with empty local spectrum
\cite[Proposition~1.2.16]{LaurNeu}. By
\cite[Proposition~1.3.2]{LaurNeu}, if $T$ has SVEP, then
$\sigma(T)$ is the union of all local spectra $\sigma_T(x)$.

The local spectral radius $r_T(x)$ of $T$ at $x$ is defined as
\[
r_T(x)=\limsup_{n\to\infty}\Vert T^n x\Vert^{1/n}.
\]
If $T$ has SVEP, then by \cite[Proposition~3.3.13]{LaurNeu}, the local spectral radius formula
\begin{equation}\label{eq:localformulawithlimsup}
r_T(x)=\max\{|z| : z\in\sigma_T(x)\}\qquad(x\in X),
\end{equation}holds, and by \cite{Vrb, Vasipaper, PruPu}, the set of those $x$ in $X$ for
which $\sigma_T(x)\neq \sigma(T)$ is then of the first category
in $X$. By \cite[Proposition~3.3.14]{LaurNeu}, it is always
true, also in the absence of SVEP, that the set of $x\in X$ for
which $r_T(x)$ is equal to the spectral radius of $T$ is of the
second category in $X$. If $T$ has Bishop's property ($\beta$)
(see \cite[Definition~1.2.5]{LaurNeu} -- it is immediate that
property ($\beta$) implies SVEP), then, by
\cite[Proposition~3.3.17]{LaurNeu},
$r_T(x)=\lim_{n\to\infty}\Vert T^n x\Vert^{1/n}$ for all $x$ in
$X$, so that the local spectral radius formula
\eqref{eq:localformulawithlimsup} holds in a stronger form as
\begin{equation}\label{eq:localformulawithlim}
\lim_{n\to\infty}\Vert T^n x\Vert^{1/n}=\max\{|z| : z\in\sigma_T(x)\}\qquad(x\in X).
\end{equation}
Furthermore, a decomposable (see
\cite[Definition~1.1.1]{LaurNeu}) operator has property
($\beta$) by \cite[Theorem~1.2.7]{LaurNeu}, hence
\eqref{eq:localformulawithlim} holds for decomposable
operators.

Now let $X=C(\Lambda)$, where $\Lambda$ is a compact Hausdorff
space, and let $T$ be the operator corresponding to
multiplication by a fixed function $g$. It is shown in
\cite[Example~1.2.11]{LaurNeu} that $T$ is decomposable, and
that $\sigma_T(f)=g(\supp f)$, for all $f\in C(\Lambda)$. As a
consequence of the general theory above, we then know that
\eqref{eq:localformulawithlim} is valid in the form
\begin{equation}\label{eq:localformulacompactHd}
\lim_{n\to\infty}\Vert g^n f\Vert_\infty^{1/n}=\max\{|g(\lambda)| : \lambda\in\supp f\}\qquad(f\in C(\Lambda)).
\end{equation}
Naturally this is also easily verified directly. The point is,
however, that local spectral theory tells us without further
computation that \eqref{eq:localformulacompactHd} holds once
the decomposability has been established and the local spectrum
has been determined.

This example makes the statement in
Theorem~\ref{thm:limSchwartz},
\begin{equation}\label{eq:localformulaSchwartz}
\lim_{n\to \infty}\| P(\partial )^n f\| _p ^{1/n} = \sup\{|P(i\lambda)| : \lambda\in\supp \F f\}\qquad(f\in\S (\R^d)),
\end{equation}
and its analogue in Theorem~\ref{thm:limLp} more or less
plausible, since under the Fourier transform $P(\partial)$
corresponds to multiplication by $P(i\lambda)$. In fact, it
raises the question whether Theorems~\ref{thm:limSchwartz}
and~\ref{thm:limLp} can conceptually be interpreted as
statements in local spectral theory for unbounded operators.
Are the right hand sides in these theorems the maximum modulus
in the local spectrum of an unbounded operator, and is there,
in addition, an a priori result available such as
\eqref{eq:localformulawithlimsup} or
\eqref{eq:localformulawithlim}, resulting in a ``explanation'' of
these theorems? The local spectral theory for unbounded
operators is not as extensive as for the bounded case, and
neither general monographs, such as \cite{Vasilescu} and
\cite{ErdWang}, nor papers dedicated to local spectral theory
for constant coefficient differential operators, such as
\cite{AlRi90,AlRi95,AlRi96, AlRi98}, contain a description of
the local spectrum of an operator $P(\partial)$, or a priori results analogous to
\eqref{eq:localformulawithlimsup} or \eqref{eq:localformulawithlim}. The only result which is
helpful in interpreting \eqref{eq:localformulaSchwartz} that we
are aware of is the following: If one can prove that
$\{P(i\lambda) : \lambda\in\supp \F f\}$ is dense in the local
spectrum  at $f$ of a closed operator in $L^p(\R^d)$ which
agrees with $P(\partial)$ on $\S (\R^d)$, and if this set is
bounded, then \cite[Proposition~IV.3.10]{Vasilescu} yields that
\[
\sup\{|P(i\lambda)| : \lambda\in\supp \F f\} \leq  \limsup_{n\to \infty}\| P(\partial )^n f\| _p ^{1/n} < \infty .
\]
Even though this is only a partial and conditional result, it
supports a conjecture that \eqref{eq:localformulaSchwartz} is
in reality a local spectral radius formula.

We can in fact prove this conjecture for $p=1$, and we will now
proceed to do so. The definition of the local resolvent set and
SVEP in the unbounded case are as above, with the additional
requirement that $\phi_z$ is in the domain of the operator for
all $z$ in $U$.

Let $P$ be a polynomial. Consider, for $1\leq p\leq\infty$, the
associated operator $T_p:C_c^\infty(\R^d)\to L^p(\R^d)$,
defined as $T_p f=P(\partial)f$, for $f\in C_c^\infty(\R^d)$.
Then, by \cite[Section~4.2]{Schechter}, $T_p$ has a closed
extension $\widetilde T_p$ on $L^p(\R^d)$, with domain
$D_{\widetilde T_p}$ equal to all $f\in L^p(\R^d)$
such that $P(\partial)f\in L^p(\R^d)$, and defined as $\widetilde T_p f=P(\partial)f$ on its domain.\footnote{By
\cite[Theorem~4.2.1]{Schechter}, $\widetilde T_p$ is actually
the closure of $T_p$, if $1\leq p<\infty$.}

\begin{lemma}\label{lem:SVEP}
$\widetilde T_1$ has SVEP.
\end{lemma}

\begin{proof}
Suppose $\phi: U\to D_{\widetilde T_1}$ is analytic, and such that $(\widetilde T_1-z)\phi_z=0$, for
all $z$ in some open non-empty $U\subset\C$. Taking Fourier transforms yields $(P(i\lambda)-z)\F \phi_z(\lambda)=0$,
for all $z\in U$ and $\lambda\in\R^d$. Thus, for any fixed $\lambda\in\R^d$, we conclude that $\F \phi_z(\lambda)=0$
for all $z\in U$ with at most one exception, and hence, since $\F \phi_z(\lambda)$ depends continuously on $z$, we
see that $\F \phi_z(\lambda)=0$ for all $z\in U$.
\end{proof}

Note that we only needed the continuity of $\phi$ in the proof
of the preceding lemma, but that it was essential that the
Fourier transform of an element in the domain of $\widetilde
T_1$ is a continuous function.

Although SVEP is strictly speaking not necessary for the
interpretation of \eqref{eq:localformulaSchwartz} as a local
spectral radius formula, the following result is indispensable.
Here we use again that the Fourier transforms are continuous
functions, but not even the continuity of $\phi$ is needed.

\begin{lemma}\label{lem:localresolventissubset}
If $f\in D_{\widetilde T_1}$, then
\[
\rho_{\widetilde T_1}(f)\subseteq\C\setminus\overline{\{P(i\lambda) : \lambda\in\supp \F f\}}.
\]
\end{lemma}
Here, and in the sequel, if $A$ is a subset of the complex numbers, then $\overline A$ denotes its closure (and not its conjugate).

\begin{proof}
Suppose that $z_0\in \rho_{\widetilde T_1}(f)$, that $U$ is an
open neighborhood of $z_0$ in $\C$ and that $(\widetilde T_1 -
z)\phi_z=f$ for some analytic $\phi:U\to D_{\widetilde T_1}$.
Taking Fourier transforms yields $(P(i\lambda)-z)\F
\phi_z(\lambda)=\F f(\lambda)$ for all $z\in U$ and
$\lambda\in\R^d$. We conclude from this that, if
$\lambda_0\in\R^d$ is such that
$P(i\lambda_0)\in\rho_{\widetilde T_1}(f)$, then $\F
f(\lambda_0)=0$. A moment's thought shows that this implies the statement in the lemma.
\end{proof}

For general $p$ and $f\in\S (\R^d)$, we can establish the reverse inclusion.

\begin{lemma}\label{lem:localresolventcontains}
Let $1\leq p\leq\infty$, and $f\in\S (\R^d)\subseteq D_{\widetilde T_p}$. Then
\[
\rho_{\widetilde T_p}(f)\supseteq\C\setminus\overline{\{P(i\lambda) : \lambda\in\supp \F f\}}.
\]
\end{lemma}

\begin{proof}
Suppose $z_0\notin \overline{\{P(i\lambda) : \lambda\in\supp \F
f\}}$. Then there exists an open neighborhood $U$ of $z_0$ in
$\C$ such that, for each $z\in U$, the function
$\psi_z:\R^d\to\C$ which is given as
\begin{equation*}
\psi_z(\lambda)=
\begin{cases}
\frac{\F f(\lambda)}{P(i\lambda)-z}&\textup{if }\lambda\in\supp \F f;\\
0 &\textup{if }\lambda\notin\supp \F f,
\end{cases}
\end{equation*}
is well defined. It is then in fact in $\S (\R^d)$.
Furthermore, it is routine to verify that the map $\psi: U\to\S
(\R^d)$, obtained by sending $z$ to $\psi_z$, is analytic on
$U$ when $\S (\R^d)$ is equipped with its usual Fr{\'e}chet
topology. Since the inverse Fourier transform is a
homeomorphism of $\S (\R^d)$, and since the inclusion map of
$\S (\R^d)$ into $L^p(\R^d)$ is continuous, the map $\phi: U\to
D_{\widetilde T_p}$, defined as $\phi_z=\F^{-1}\psi_z$ for
$z\in U$, is also analytic on $U$. By construction it satisfies
$(\widetilde T_p -z)\phi_z=f$ for $z\in U$. We conclude that
$z_0\in \rho_{\widetilde T_p}(f)$.
\end{proof}

Combining the above lemmas, we have the following satisfactory result for $p=1$.

\begin{corollary}\label{cor:interpretation}
The closed operator $\widetilde T_1$ on $L^1(\R^d)$ has SVEP. Furthermore, if $f\in\S (\R^d)\subseteq D_{\widetilde T_1}$, then
\[
\sigma_{\widetilde T_1}(f)=\overline{\{P(i\lambda) : \lambda\in\supp \F f\}}.
\]
Hence the case $p=1$ of Theorem~\ref{thm:limSchwartz}, when read in the equivalent form
\begin{equation}\label{eq:equivalentform}
\lim_{n\to \infty}\| {\widetilde T}_1^n f\| _1 ^{1/n} = \sup\left\{|z| : z\in\overline{\{P(i\lambda) : \lambda\in\supp \F f\}}\right\}\qquad(f\in\S (\R^d)),
\end{equation}
expresses the validity of the local spectral radius formula
\[
\lim_{n\to \infty}\| {\widetilde T}_1^n f\| _1 ^{1/n} = \sup\{|z| : z\in\sigma_{\widetilde T_1}(f)\}
\]
in the extended positive real numbers.
\end{corollary}

While establishing this result we have exploited, in particular
in the proof of Lemma~\ref{lem:localresolventissubset}, that
the Fourier transform maps $L^1(\R^d)$ into a space of
distributions consisting of continuous functions. Whereas for
$1<p\leq 2$, the image still consists of equivalence classes of
functions, so that pointwise arguments as the above can perhaps
be adapted, it is known that for $2<p\leq\infty$ the image of
$L^p(\R^d)$ contains distributions of strictly positive order,
see \cite[Theorem~7.6.6]{Hoer}.\footnote{An upper bound for the
order is known, see \cite[p.~242]{Hoer}.} Therefore it seems
likely that a more refined analysis will be necessary to
investigate the general case, using a priori knowledge about
the elements of $D_{\widetilde T_p}$ and also the analyticity
of the local resolvent functions, which we did not use in the
proof of Lemma~\ref{lem:localresolventissubset} at all, as
possible extra ingredients.

In view of the above we feel confident enough to propose the following.
\begin{conjecture}
Theorems~\ref{thm:limSchwartz} and~\ref{thm:limLp}, when read as in \eqref{eq:equivalentform},
are local spectral radius formulas.
\end{conjecture}

In conclusion we mention that this conjecture is also supported by the results for compact connected
Lie groups in \cite{AdJ2}, where, for arbitrary $p$ and smooth functions on the group, not only the analogue
of Theorem~\ref{thm:limSchwartz} is established, but also the interpretation, as in Corollary~\ref{cor:interpretation},
of this analogous result as a local spectral radius formula.

\end{document}